\documentclass[11pt,a4paper]{amsart}
\usepackage[headings]{fullpage}
\usepackage{amssymb,amscd,hyperref}
\usepackage{amsthm}
\usepackage[alphabetic, nobysame]{amsrefs}
\usepackage[all,cmtip]{xy}
\usepackage{color}
\usepackage{xcolor}
\usepackage{tikz-cd} 
\numberwithin{equation}{section}
\newtheorem{thm}{Theorem}[section]
\newtheorem{lem}[thm]{Lemma}
\newtheorem{prop}[thm]{Proposition}
\newtheorem{cor}[thm]{Corollary}

\theoremstyle{definition}
\newtheorem{defn}[thm]{Definition}

\theoremstyle{remark}
\newtheorem{rem}[thm]{Remark}

\def\Z{\mathbb{Z}}

\def\F{\mathbb{F}}

\def\THR{\mathsf{THR}}
\def\ihh{\mathsf{iHH}}
\def\hr{\mathsf{HR}}
\def\hh{\mathsf{HH}}
\def\CH{\mathsf{CH}}

\def\Tor{\mathsf{Tor}}

\def\tamb{\text{-Tamb}}

\def\fix{\mathrm{fix}}
\def\Af{\underline R^\fix}
\def\Mf{\underline M^\fix}
\def\Nf{\underline N^\fix}
\newcommand{\und}[1]{{\underline{#1}}}
\def\norm{\mathsf{norm}}
\def\trace{\mathsf{tr}}
\def\tr{\mathsf{tr}}
\def\res{\mathsf{res}}

\def\ra{\rightarrow}

\def\cL{\mathcal{L}}

\def\HH{\mathsf{HH}}

\def\id{\mathrm{id}}
\def\op{\mathrm{op}}
\def\cone{\mathrm{cone}}

\begin{document}
\title[Reflexive \& involutive Hochschild homology as equivariant Loday constructions]{Reflexive homology and involutive Hochschild homology as
  equivariant Loday constructions}

\author{Ayelet Lindenstrauss}
\address{Mathematics Department, Indiana University, 831 East Third Street,
  Bloomington, IN 47405, USA}
\email{alindens@iu.edu}

\author{Birgit Richter}
\address{Fachbereich Mathematik der Universit\"at Hamburg,
  Bundesstra{\ss}e 55, 20146 Hamburg,  Germany}
\email{birgit.richter@uni-hamburg.de}

\date{\today}

\begin{abstract}
  For associative rings with anti-involution several homology theories exist,
  for instance reflexive homology as studied by Graves and 
involutive Hochschild homology defined by Fern\`andez-Val\`encia  and
Giansiracusa. We prove that the corresponding homology groups can be
identified with the homotopy groups of an equivariant Loday construction 
of the one-point compactification of the sign-representation evaluated at the
trivial orbit, if we assume that $2$ is invertible and if the underlying
abelian group of the ring is flat. We also show a relative version where we
consider an associative $k$-algebra with an anti-involution where $k$ is an
arbitrary commutative ground ring. 
  
\end{abstract}
\maketitle  

\section{Introduction}

In \cite{lrz-gloday} we introduced equivariant Loday
constructions. These generalize the non-equi\-variant Loday constructions, which include  (topological) Hochschild homology, higher order Hochschild homology
and torus homology.

In the equivariant case we fix a finite group $G$. The starting
point for a Loday construction is a $G$-commutative monoid in the 
sense of Hill and Hopkins \cite{hill-hopkins}. In the setting of
$G$-equivariant stable homotopy theory these are genuine
$G$-commutative ring spectra 
whereas in the algebraic setting of Mackey functors $G$-commutative
monoids are $G$-Tambara functors. Some equivariant homology theories
such as 
the twisted cyclic nerve of Blumberg-Gerhardt-Hill-Lawson \cite{bghl}
and Hesselholt-Madsen's 
Real topological Hochschild homology, $\THR$, \cite{dmpr} can be
identified with such equivariant Loday constructions \cite[\S
7]{lrz-gloday}. Here, $\THR$ is a homology theory for 
associative algebra spectra with anti-involution $A$ and we identified
this in the commutative case with the Loday construction over the one-point compactification of
the sign-representation,
$\THR(A) \simeq \cL^{C_2}_{S^\sigma} (A)$. 
In the following we will often refer to $S^\sigma$ as the \emph{flip circle}. 
In \cite[Proposition 6.1]{lrz-gloday}, we show that for 
any $G$-simplicial set $X$, if we apply the functor $\underline{\pi}_0$ levelwise to the equivariant
Loday construction of a connective genuine commutative $G$-algebra spectrum $A$ to obtain 
a simplicial $G$-Tambara functor, 
\[\underline{\pi}_0 (\cL_X^G(A)) \cong \cL^G_X(\underline{\pi}_0 (A)), \] 
 which relates $\cL^{C_2}_{S^\sigma}$ of $C_2$-Tambara functors to $\THR$.

There is an algebraic version of $\THR$, called Real Hochschild
homology \cite[Definition 6.15]{akgh} that takes associative algebras
with anti-involution as input. These are associative $k$-algebras for
some commutative ring $k$, such that $\tau(a) := \bar{a}$ satisfies
$\overline{ab} = \bar{b} \bar{a}$ and such that the $C_2$-action is
$k$-linear.  
In her thesis Chloe Lewis developed a B\"okstedt-type
spectral sequence for $\THR$ \cite{lewis} whose $E^2$-term consists of
Real Hochschild homology groups. Other homology theories for
associative algebras with anti-involution 
are reflexive homology \cite{graves} and involutive Hochschild
homology \cite{fvg}. Reflexive homology is a homology theory
associated to the crossed simplicial group that is 
the cyclic group of order two, $C_2 = \langle \tau\rangle$, in every
simplicial degree, where we do \emph{not} view $C_2$ as a constant
simplicial group, but let $\tau$ interact with the category $\Delta$ by
reversing the  
simplicial structure. Involutive Hochschild homology was defined in
\cite{fvg}; the corresponding cohomology theory was developed by Braun
\cite{braun}, who defined a 
cohomology theory for involutive $A_\infty$-algebras, motivated by
work of Costello on open Klein topological conformal field theories
\cite{costello}. We slightly generalize the definition in \cite{fvg} and work
over arbitrary commutative rings instead of fields.

We prove the identification of reflexive homology, $\hr_*$, with the
homotopy groups of an equivariant Loday construction in section
\ref{sec:hr} and the one for involutive Hochschild homology, $\ihh_*$,
in section \ref{sec:ihh}: 

\bigskip

\textbf{Theorem} (Theorems \ref{thm:lodayishr} and \ref{thm:ihh}) \, Assume that $R$
is a commutative ring with involution and that $2$ is invertible in $R$. If the underlying abelian group of $R$ is flat, then
\[ \ihh_*^\Z(R) \cong \pi_*(\cL^{C_2}_{S^\sigma}(\und{R}^\fix )(C_2/C_2)) \cong \hr^{+,\Z}_*(R,R). \] 
Here, $\cL^{C_2}_{S^\sigma}(\und{R}^\fix )$ is the $C_2$-equivariant
Loday construction of the fixed point Tambara functor for $R$,
$\Af$, for the representation sphere of the real sign-representation,
$S^\sigma$. This is a simplicial Tambara functor and 
$\cL^{C_2}_{S^\sigma}(\und{R}^\fix )(C_2/C_2)$ is its evaluation at
the trivial orbit $C_2/C_2$. This yields a simplicial abelian group
and we consider its homotopy groups.

If we work relative to a commutative ground ring $k$, then we obtain a corresponding result:

\bigskip

\textbf{Theorem} (Theorems \ref{thm:lodayishrk} and \ref{thm:ihhk}) \, 
Assume that $R$ is a commutative $k$-algebra with a $k$-linear
involution and that $2$ is 
invertible in $R$. If the underlying module of $R$ is flat over $k$,
then
\[ \ihh_*^k(R) \cong \pi_*(\cL^{C_2
    ,\und{k}^c}_{S^\sigma}(\und{R}^\fix)(C_2/C_2)) \cong 
  \hr^{+,k}_*(R,R).   \]  
In hindsight, this identifies the Loday construction over the $C_2$-Burnside
Tambara functor with the Loday construction relative to $\und{\Z}^c$ under the
above  assumptions (see Remark \ref{rem:znota}). We consider the examples of $\F_2$ and $\Z$ with
the trivial $C_2$-action in section  \ref{sec:exs} in order to
understand what happens if we drop these assumptions. There, the
homotopy groups of the Loday constructions
 \emph{differ} both from reflexive homology and from  involutive Hochschild homology.

The relationship to
the Real Hochschild homology of \cite{akgh} is more subtle: The latter
takes all dihedral groups into account and for $D_2 = C_2$ their definition
agrees with our equivariant Loday construction. We will establish a
full comparison also for the higher $D_{2m}$ with equivariant
Loday constructions in future work with Foling Zou (in preparation). 

\bigskip

In section \ref{sec:Green} we extend our results to the associative case,
where we consider
associative rings $R$ and associative $k$-algebras with anti-involution where
$k$ is an arbitrary commutative ground ring. Usually, one cannot form Loday
constructions without assuming commutativity, but the simplicial model of the
one-point compactification of the sign-representation consists of two glued
copies of the simplicial $1$-simplex with its intrinsic ordering, so we can
extend the definition to equivariant associative monoids in this case and we
get results generalizing the above theorems:

\bigskip

\textbf{Theorem} (Theorem \ref{thm:anti-invring}) \, Assume that $R$ 
is an associative ring with anti-involution and that $2$ is invertible in $R$. If the underlying abelian group of $R$ is flat, then
\[ \ihh_*^\Z(R) \cong \pi_*( \cL^{C_2}_{S^\sigma}(\und{R}^\fix )(C_2/C_2)) \cong \hr^{+,\Z}_*(R,R). \] 

If we work relative to a commutative ground ring $k$, then we obtain a corresponding result:

\bigskip

\textbf{Theorem} (Theorem \ref{thm:anti-invalg}) \, 
Assume that $A$ is an associative $k$-algebra with a $k$-linear
anti-involution and that $2$ is 
invertible in $A$. If the underlying module of $A$ is flat over $k$,
then
\[ \ihh_*^k(A) \cong \pi_*(\cL^{C_2 ,\und{k}^c}_{S^\sigma}(\und{A}^\fix)(C_2/C_2))
  \cong   \hr^{+,k}_*(A,A).   \]

The proofs, however, are different: In the case of an associative ring $R$ with
anti-involution the fixed point Mackey functor $\Af$ is
not an associative $C_2$-Green functor, so in particular it is not a
$C_2$-Tambara functor. But it has the structure of a discrete
$E_\sigma$-ring in the sense of \cite[\S 6.3]{akgh} and it is also 
a Hermitian Mackey
functor in the sense of \cite[Definition 1.1.1]{do}. Real Hochschild
homology \cite[\S 6.4]{akgh} and Real topological Hochschild homology
\cite[Example 2.4]{dmpr} are
defined for such objects, see also \cite[Proposition 7.1.1]{horev} for
the analogous result for factorization 
homology of the flip circle $S^\sigma$, so it is not surprising that
one can extend the Loday construction for $S^\sigma$ to fixed
point Mackey functors of rings with anti-involution. 
We endow the 
$C_2$-Mackey norm functor of $i_e^*\Af$ with the structure of an
associative Green functor so that  $\Af$ is a bimodule over it. 

\subsection*{Acknowledgements} 
We thank Sarah Whitehouse and Jack Davidson for helpful comments. 

The first author was supported by NSF grant DMS-2004300 at the beginning of working on this paper,
and by a grant from the Simons Foundation (\#359565, Ayelet Lindenstrauss) afterwards. The second author thanks the Department of Mathematics of
the Indiana University Bloomington for its hospitality and
support. An anonymous referee helped us to improve the paper with many 
insightful comments.  
\section{Equivariant Loday constructions}

We recall the basic facts about equivariant Loday constructions for
$G$-Tambara functors from \cite{lrz-gloday} for an arbitrary finite
group $G$. We work with unital rings. We assume that ring maps
preserve the unit, and that the unit acts as the identity on any module over the
ring.

We consider simplicial $G$-sets $X$ that are
finite in every degree and call them finite simplicial
$G$-sets. For every $G$-Tambara functor $\und{T}$ and every such $X$
the simplicial $G$-Tambara functor 
$\cL_X^G(\und{T})$ is the $G$-Loday construction for $X$ and
$\und{T}$. In simplicial degree $n$ we define:  
\[ \cL_X^G(\und{T})_n = X_n \otimes \und{T}\]
where the formation of the tensor product with the finite $G$-set $X_n$ uses
the fact that $G$-Tambara functors are the $G$-commutative monoids in
the setting of $G$-Mackey functors. This was proved by Mazur
\cite{mazur} for cyclic $p$-groups for a prime $p$ and by Hoyer
\cite{hoyer} in the case of a general finite group $G$. As they show
that the construction $X_n \otimes \und{T}$ is functorial in $X_n$,
the Loday construction is well-defined.

The above tensor can be made explicit. Every finite $G$-set is
isomorphic to a finite disjoint union of orbits and Mazur and Hoyer
show that for an orbit $G/H$ we obtain
\[ G/H \otimes \und{T} \cong  N_H^Gi_H^*\und{T}. \]
Here, $i_H^*$ restricts a $G$-Tambara functor to $H$, so for a finite
$H$-set $Y$, $i_H^*\und{T}(Y) := \und{T}(G \times_H Y)$. The restriction functor
has the norm functor $N_H^G$ as a left adjoint. A disjoint union of $G$-sets
$X, X'$, $X \sqcup X'$ is sent to
\[ (X \sqcup X') \otimes \und{T} \cong (X \otimes \und{T}) \Box
  (X' \otimes \und{T}), \] 
so this determines every $X_n \otimes \und{T}$ up to isomorphism. 

\section{Basic results about fixed point Tambara functors}
In this section we study $C_2$-Mackey and Tambara functors.
If $L$ is an abelian group with involution 
$a\mapsto \bar a$, there is a $C_2$-Mackey functor $\und{L}^\fix$
given by 
\[  \und{L}^\fix =
  \begin{cases}
   L^{C_2} & \text{ at } C_2/C_2, \\
    L & \text{ at } C_2/e,
  \end{cases}
  \] 
where $\trace(a)=a+\bar a$ for all $a\in L$ and $\res(a)=a$ for all
$a\in L^{C_2}$. 
If $R$ is a commutative ring whose multiplication is compatible with
its involution, then we can define $\norm(a)=a\bar a$ and get a
$C_2$-Tambara functor structure on $\und{R}^\fix$. 

\begin{rem}\label{adjoint}
Note that for an arbitrary finite group $G$, the functor that sends a
commutative $G$-ring $T$ to its fixed point $G$-Tambara functor is right
adjoint to the functor that takes a $G$-Tambara functor $\und{R}$ to
its underlying commutative $G$-ring $\und{R}(G/e)$ (see for instance
\cite[Lemma 2.9]{ssw}). This is completely analogous to the situation in Mackey functors, where the
functor that sends an abelian group $A$ with $G$-action to its fixed
point $G$-Mackey functor $\und{A}^\fix$  is right adjoint to the functor that takes a $G$-Mackey functor $\und{M}$ and sends it to the abelian group with $G$-action $\und{M}(G/e)$.
\end{rem} 

For $G=C_2$ a description of the counit map
$\varrho_T
\colon \und{T}^\fix(C_2/e) \ra T$ and the unit map $\eta_\und{R} \colon \und{R} \ra
\und{\und{R}(C_2/e)}^\fix$ of this adjunction are very straightforward (both in the Mackey case and in the Tambara case): The counit $\varrho_T \colon  \und{T}^\fix(C_2/e) = T \ra T$ is the identity
map.

At the
free orbit
\[ \eta_\und{R}(C_2/e) \colon \und{R}(C_2/e) \rightarrow
  \und{\und{R}(C_2/e)}^\fix(C_2/e) = \und{R}(C_2/e)\]
is the identity map and at the trivial orbit $C_2/C_2$
\[ \eta_\und{R}(C_2/C_2) \colon \und{R}(C_2/C_2) \rightarrow
  \und{\und{R}(C_2/e)}^\fix(C_2/C_2) = \und{R}(C_2/e)^{C_2}\]
is the restriction map.

For an arbitrary finite group $G$ this adjunction ensures that morphisms of
commutative $G$-rings $f \colon R \ra T$ are in bijective
correspondence with  morphisms of
$G$-Tambara functors $\und{f} \colon \Af \ra \und{T}^\fix$, because
$\Af(G/e) = R$.

In
particular if $k$ is a
commutative ring and if $f \colon R \ra T$ is a morphism of
commutative $k$-algebras with involution, then we get a commutative
diagram of $C_2$-Tambara functors 
\[ \xymatrix{ & \und{k}^c \ar[dr]^{\und{i_T}} \ar[dl]_{\und{i_R}}&  \\
    \Af \ar[rr]^{\und{f}} & & \und{T}^\fix}\]
where $i_R$ and $i_T$ are the unit maps of $R$ and $T$. Here
$\und{k}^c$ denotes the fixed point Tambara for the trivial
$C_2$-action, called the constant Tambara functor.

For a set $Y$ we denote by $\Z\{Y\}$ the free abelian group
generated by $Y$ and for $y \in Y$ the corresponding generator in
$\Z\{Y\}$ is $\{y\}$. When $R$ is a commutative ring with involution
the norm restriction of $\und{R}^\fix$ is given by  
 
\[     N^{C_2}_ei^{*}_e \Af = N_e^{C_2}R =
   \begin{cases}
 ( \Z\{ R\} \oplus (R\otimes R)/C_2)/\mathrm{TR} & \text{ at } C_2/C_2 \\
   R\otimes R & \text{ at } C_2/e,
  \end{cases}
  \] 
  where $C_2$ acts on $R\otimes R$ via $\tau(a\otimes b) = \bar
  b\otimes \bar a$, $[a\otimes b]$ denotes the equivalence class of
  $a\otimes b$ in $ (R\otimes R)/C_2 $,  and Tambara Reciprocity, TR,
  identifies  
  $\{a+b\} \sim \{a\} +\{b\} + [a\otimes \bar b]$.  Here
  $\norm(a\otimes b)=\{a\bar b\}$ and $\trace(a\otimes b)=[a\otimes b]$ for
  all $a\otimes b\in R\otimes R$, $\res(\{a\})=a\otimes \bar a$, and
  $\res([a\otimes b])=a\otimes b+\bar b\otimes \bar a$  (see \cite{hm}
  for properties of the norm functor, especially Fact 4.4 in loc.~cit.).

  \begin{lem}\label{lem:fixprod}
Assume that $M$ and $N$ are two abelian groups with involution and
assume that $2$ is invertible in $M$ or in $N$. Then there is an equivalence of
$C_2$-Mackey functors  
\[\Mf\Box\Nf \cong \und{(M\otimes N)}^\fix\]
which is natural in $M$ and $N$. Here $C_2$ acts on $M\otimes N$ by
the diagonal action.  If, in addition, $M$ and $N$ are both
commutative rings with 
involution, then $\Mf$, $\Nf$, and $\und{(M\otimes N)}^\fix$ are
$C_2$-Tambara functors, and the above equivalence is an equivalence of
$C_2$-Tambara functors. 
\end{lem}

\begin{proof}
We have by the definition of the box product (see e.g.~\cite[Definition 3.1]{hm})
\[
\Mf\Box\Nf =
\begin{cases}
[M^{C_2} \otimes N^{C_2} \oplus (M\otimes N)/C_2]/\text{FR} &\text{at }C_2/C_2\\
M\otimes N &\text{at }C_2/e,
\end{cases}
\]
and we map it to $\und{(M\otimes N)}^\fix$ by using the Mackey or Tambara functor map corresponding to the identity map on the free orbit $C_2/e$ via the adjunction in Remark \ref{adjoint}.   Then our map will consist of the identity on the free orbit and $\res$ on the trivial orbit, which will land in $(M\otimes N)^{C_2}$.  

On the trivial orbit the map $\res$ includes $M^{C_2} \otimes
N^{C_2}$ into $(M\otimes N)^{C_2}$ and maps $(M\otimes N)/C_2$ to $(M\otimes N)^{C_2}$ by the map $[m\otimes n]\mapsto m\otimes n +\bar m\otimes\bar n$ (which is an equivalence since $2$ is invertible in $M\otimes N$).

Then $\res$ maps $\Mf\Box\Nf (C_2/C_2)$ surjectively onto $ (M\otimes
N)^{C_2}$ because its restriction to the second summand does, and it
is injective because its restriction to the second summand is
injective and Frobenius Reciprocity allows us to identify the first
summand into the second one: if $2$ is invertible in $M$, any $m\in
M^{C_2}$ is equal to $\trace (m/2)$ so $m\otimes n$ is identified with
$[m/2\otimes n]$, and if $2$ is invertible in $N$ similarly $m\otimes
n$ is identified with $[m\otimes n/2]$.
\end{proof}

\begin{lem}\label{lem:prodnorm}
If $R$ is a  commutative ring with involution in which $2$ is
invertible and if $M$ is an abelian group with an involution, then there
is an equivalence of $C_2$-Mackey functors 
\begin{equation} \label{eq:normmodule} (N^{C_2}_ei_e^*\Af) \Box \Mf =  (N^{C_2}_eR) \Box \Mf \cong \und{(R \otimes  R \otimes
    M)}^\fix, \end{equation}
which is natural in $M$ and $R$. Here, $C_2$ acts on $R \otimes R
\otimes M$ by $\tau(a\otimes b \otimes m)=\bar b\otimes \bar a\otimes
\bar m$. 
If $M$ is also a commutative ring with involution, then
\eqref{eq:normmodule}  is an equivalence of $C_2$-Tambara functors. 
\end{lem}
\begin{proof}
Applying the formula for the box product (\cite[Definition 3.1]{hm}) to the formula for the norm yields
\[     (N^{C_2}_eR) \Box \Mf =
   \begin{cases}
\bigl( ( \Z\{R\} \oplus (R\otimes R)/C_2 ) /\mathrm{TR}  \otimes M^{C_2} \oplus (R\otimes R \otimes M)/C_2 \bigr)/\mathrm{FR}& \text{ at } C_2/C_2 \\
   R \otimes R \otimes M& \text{ at } C_2/e,
  \end{cases}
  \] 
 and again we send it to $\und{(R \otimes  R \otimes
    M)}^\fix$ by the Mackey or Tambara functor map that corresponds via the adjunction of Remark \ref{adjoint} to the identity on the free orbit.  So the resulting map is the identity on the free orbit  and $\res$ on the trivial orbit, which will land in $(R \otimes R \otimes M)^{C_2}$.
 
Again, since $2$ is invertible the restriction of $\res$ to $(R\otimes R \otimes M)/C_2$, which sends $[a\otimes b\otimes m]  \mapsto a\otimes b \otimes m + \bar b\otimes
                        \bar a\otimes \bar m $
is an isomorphism $(R\otimes R \otimes M)/C_2 \to (R\otimes R \otimes M)^{C_2} $.
Frobenius Reciprocity identifies 
$\trace(a\otimes b)\otimes m$ with $[a\otimes b\otimes
  \res(m)]\in (R\otimes R \otimes M)/C_2$ for all $a, b \in R$, $m\in M^{C_2}$.
  It also identifies  $\{ a\} \otimes \trace(m)$ with $[a\otimes \bar a \otimes m ] \in (R\otimes R \otimes M)/C_2$ for all $a\in R$, $m\in M$, and since $2$ is invertible, $\trace(M) =M^{C_2}$ (any $m\in M^{C_2}$ is equal to $\trace (m/2)$).
  So in fact, 
  \[(N^{C_2}_eR) \Box \Mf (C_2/C_2) \cong (R \otimes R \otimes
    M)/{C_2} \cong (R \otimes R \otimes M)^{C_2} \] 
  and our map between  $(N^{C_2}_eR)\Box \Mf $ and $\und{(R \otimes
    R\otimes  M)}^\fix$ is an isomorphism at both levels.   \end{proof}

Note also that for a map of commutative $C_2$-rings
$f\colon  R \to C$ where $2$ is invertible in both rings, the sequence of maps 
\[N^{C_2}_ei^{*}_e \Af \Box \und{C}^\fix \to \Af\Box \und{C}^\fix
  \cong \und{(R \otimes
    C)}^\fix \to \und{(C\otimes C)}^\fix \to \und{C}^\fix \] 
that is given by the counit of the $(N_e^{C_2},i^*_e)$-adjunction, the
identification of 
Lemma \ref{lem:fixprod}, the map $f$, and multiplication, the
corresponding map $\und{(R \otimes R \otimes C)}^\fix 
\to \und{C}^\fix$ is given by the multiplication in $R$ and the
$R$-module structure on $C$ induced by the map
$f$.

\section{Working relative to a commutative ground ring}

In \cite[\S 8]{lrz-gloday} we defined a $G$-equivariant Loday
construction relative to  a map of $G$-Tambara functors $\und{k} \ra
\und{R}$. In general, this construction is rather involved because its
building blocks are relative norm-restriction terms: For an orbit
$G/H$ we set 
\begin{equation} \label{eq:relnorm} G/H \otimes_{\und{k}} \und{R} := (G/H \otimes \und{R} ) \Box_{(G/H \otimes{\und{k}} )}  \und{k}  =N_H^Gi_H^*(\und{R})
  \Box_{N_H^Gi_H^*(\und{k})} \und{k} =: N_H^{G,\und{k}}i_H^*(\und{R}). \end{equation}
This uses the naturality of $N_H^Gi_H^*(-)$ and the counit
$N_H^Gi_H^*(\und{k}) \ra \und{k}$ of the norm-restriction adjunction.

In \cite{lrz-gloday} we define the relative equivariant Loday
construction for any finite simplicial $G$-set  $X$: 
\[ \cL^{G, \und{k}}_X(\und{R}) := \cL^G_X (\und{R})
  \Box_{\cL^G_X(\und {k})} \und{k}.\] 

If we consider fixed point $C_2$-Tambara functors $\und{R}^\fix$ and  if we work
relative to a constant Tambara functor $\und{k}^c$, then these terms simplify
drastically.  Recall
that we denote by 
$\und{k}^c$ the constant Tambara functor which is the fixed point
Tambara functor for the trivial $C_2$-action.

The purpose of this section is to relate the relative Loday
construction of $C_2$-fixed point Tambara functors to the fixed point
Tambara functor of the non-equivariant relative Loday construction. To
that end we prove two crucial auxiliary results. 

\begin{prop} \label{prop:orbits} 
  Let $k \ra R$ be a map of commutative $C_2$-rings where $C_2$ acts
  trivially on $k$ and  $2$ is invertible in $R$.
  Then 
\[  N_e^{C_2,\und{k}^c}i_e^*(\und{R}^\fix)  \cong \und{(R \otimes_k R)}^\fix, \] 
where $C_2$ acts on $R \otimes_k R$ by $\tau(a\otimes b)=\bar b
\otimes \bar a$. 
\end{prop}

\begin{proof}  
  The relative box-product $ N_e^{C_2,\und{k}^c}i_e^*(\und{R}^\fix) = N_e^{C_2}i_e^*(\und{R}^\fix
  )\Box_{N_e^{C_2}i_e^*(\und{k}^c)} \und{k}^c$ is the coequalizer of
  the diagram 
  \[
\xymatrix@1{ N_e^{C_2}i_e^*(\und{R}^\fix) \Box N_e^{C_2}i_e^*\und{k}^c
  \Box \und{k}^c \ar@<0.7ex>[rr]^(0.55){\nu \Box \id}
  \ar@<-0.7ex>[rr]_(0.55){\id \Box \nu'} & &
  N_e^{C_2}i_e^*(\und{R}^\fix) \Box \und{k}^c} \] 
where $\nu$ is the composite of the map $N_e^{C_2}i_e^*\und{k}^c \ra
N_e^{C_2}i_e^*(\und{R}^\fix)$ and the multiplication map of
$N_e^{C_2}i_e^*(\und{R}^\fix)$ and $\nu'$ uses the counit of the
adjunction $\varepsilon \colon N_e^{C_2}i_e^*\und{k}^c \ra \und{k}^c$
and the multiplication in $\und{k}^c$. As $\und{k}^c$ is the
fixed point Tambara functor for the trivial action we can use the fact
that $i_e^*$ and $N_e^{C_2}$ are strong symmetric monoidal and Lemma
\ref{lem:fixprod} to get that  
\[N_e^{C_2}i_e^*(\und{R}^\fix) \Box N_e^{C_2}i_e^*(\und{k}^c) =
  N_e^{C_2}i_e^*(\und{R}^\fix \Box \und{k}^c) =
N_e^{C_2}i_e^*(\und{R}^\fix \Box \und{k}^\fix)
\cong N_e^{C_2}i_e^*(\und{(R\otimes k)}^\fix).\] 
Then we can use Lemma \ref{lem:prodnorm} to rewrite the diagram as
\[ \xymatrix@1{ \und{((R \otimes k) \otimes (R \otimes k) \otimes
      k)}^\fix \ar@<0.7ex>[rr]^(0.55){\nu \Box \id}
    \ar@<-0.7ex>[rr]_(0.55){\id \Box \nu'} & & \und{(R \otimes R
      \otimes k)}^\fix} \] 
where now $\nu$ uses the map $k \ra R$ and the induced $k$-module
structure on $R$ and $\nu'$ uses the multiplication in $k$.

Note that $R \otimes_k R \cong (R \otimes R) \otimes_{k \otimes k}
k$. We will show that taking the fixed point Tambara
functor commutes with forming coequalizers.

To this end we consider the full subcategory of $C_2$-Tambara functors
whose objects are fixed point Tambara functors and denote it by
$C_2\tamb^{\fix}$. But then the fact that the functor $T \mapsto
\und{T}^\fix$ is right adjoint to evaluation at the free level implies that
\begin{align*} C_2\tamb^{\fix}(\Af, \und{T}^\fix) & = C_2\tamb(\Af, \und{T}^\fix)\cong cC_2\text{-rings}(\Af(C_2/e),
                                                    T) \\
  & = cC_2\text{-rings}(R,
  T) = cC_2\text{-rings}(R, \und{T}^\fix(C_2/e)),  \end{align*}
where $cC_2\text{-rings}$ denotes the category of commutative
$C_2$-rings. Therefore, if we restrict to the above full subcategory,
taking the fixed point Tambara functor is left adjoint, hence preserves coequalizers.
\end{proof}


\begin{lem}\label{lem:assemble} 
If $k$ is a commutative ring with trivial $C_2$ action and 
$M$ and $N$ are two $k$-modules with a $k$-linear involution   and $2$
is invertible in $M$ or in $N$, then there is
an equivalence of $C_2$-Mackey functors 
\[\Mf\Box_{\und{k}^c} \Nf \cong \und{(M\otimes_k N)}^\fix\]
which is natural in $M$ and $N$. Here $C_2$ acts on $M\otimes N$ by
the diagonal action.  If $M$ and $N$ are both also commutative
$k$-algebras, this is an equivalence of $C_2$-Tambara functors. 
\end{lem}
\begin{proof}
Using Lemma \ref{lem:fixprod} we know that 
\[\Mf\Box{\und{k}^c} \Box \Nf = \Mf\Box{\und{k}^\fix} \Box \Nf \cong
  \und{(M\otimes k \otimes N)}^\fix\] 
and
\[\Mf\Box \Nf  \cong \und{(M\otimes N)}^\fix.\]
The result in the $k$-module case then follows from the fact that
taking the fixed point  Mackey functor commutes with forming
coequalizers, which is completely analogous to the fact that the
fixed point Tambara functor commutes with forming coequalizers which
was shown in the proof of Proposition \ref{prop:orbits}  
above.  In the case of $k$-algebras, it follows directly by the
argument in the proof there. 
\end{proof}

\begin{thm}\label{thm:relativeisgood}
Assume that $k \ra R$ is a map of commutative $C_2$-rings where $C_2$ acts
trivially on $k$ and $2$ is invertible in $R$, and let $X$ be a finite
simplicial $C_2$-set. Then 
\[ \cL^{C_2, \und{k} ^c}_X(\und {R}^\fix) \cong \und{\cL^k_X(R)}^\fix,\]
where $C_2$ acts on each level $\cL^k_{X_n}(R)$ by simultaneously using the action induced from the $C_2$-action on $X_n$ (exchanging copies of $R$ as needed) by naturality and acting on all copies of $R$.
\end{thm}

\begin{proof}
Theorem \ref{thm:relativeisgood} follows from the previous two results: Proposition \ref{prop:orbits} says that for free orbits $C_2/e$,
\[ \cL^{C_2, \und{k} ^c}_{C_2/e}(\und {R}^\fix) \cong \und{\cL^k_{C_2/e}(R)}^\fix.\]
Clearly
 for one-point orbits, 
 \[ \cL^{C_2, \und{k} ^c}_{C_2/C_2}(\und {R}^\fix) = \und {R}^\fix= \und{\cL^k_{C_2/C_2}(R)}^\fix.\]
If $X$ and $Y$ are disjoint $C_2$-sets
\[ \cL^{C_2, \und{k} ^c}_{X\sqcup Y} (\und {R}^\fix) \cong \cL^{C_2, \und{k} ^c}_{X} (\und {R}^\fix) 
\Box_{\und{k} ^c}
\cL^{C_2, \und{k} ^c}_{Y} (\und {R}^\fix). \]
Then Lemma \ref{lem:assemble} implies that the identification for the free
and trivial orbits can be assembled into a statement about disjoint
unions of orbits. This gives the desired identification in each fixed
simplicial degree. 

Face maps in $X$ are surjective and the identifications above are
compatible with fold maps, orbit surjections $C_2/e \ra C_2/C_2$ and
isomorphisms. As degeneracy maps just insert units, they are also
compatible with the degreewise isomorphisms.

\end{proof}

\section{Identifying $\cL_{S^\sigma }^{C_2}(\Af)$}
In this section we will continue to work with the cyclic group of order
$2$, $C_2 = \langle \tau \mid \tau^2=e\rangle$, and we will consider the
$C_2$-simplicial set $S^\sigma$  which models the 
one-point compactification of the real sign-representation, 
\[  \xymatrix@R=0.2cm{ & \bullet \\ S^\sigma = & \\ & \bullet \ar@/_3ex/[uu] \ar@/^3ex/[uu]}\]
where the $C_2$-action flips the two arcs.  We will call $S^\sigma$ with this action the flip circle.

By \cite[(7.4)]{lrz-gloday}, for any $C_2$-Tambara functor $\und{T}$
we can express the $C_2$-Loday construction of $\und{T}$ with respect
to $S^\sigma$ as a two-sided bar construction
\begin{equation}
  \label{eq:compare-flip}
\cL^{C_2}_{S^{\sigma}}(\und{T}) \cong B(\und{T}, N^{C_2}_ei^{*}_e\und{T}, \und{T}).
\end{equation}
We will simplify this for the $C_2$-Tambara functor $\und{R}^\fix$ associated to a commutative ring  $R$ with involution $a\mapsto \bar a$.    
We will repeatedly use  the commutative $C_2$-ring $R \otimes R$, with
 \begin{equation}\label{eq:algnorm}
 \tau(a\otimes b) = \bar b\otimes \bar a.
 \end{equation}

For a ring spectrum $A$ with an anti-involution, Dotto, Moi,
Patchkoria and Reeh observed \cite[p.~84]{dmpr}, that
\[ B(A, N_e^{C_2}i_e^*A, A) \simeq B(A, A \wedge A,
  A), \]
where they use the flip-$C_2$-action on $A \wedge A$ (switching
coordinates and acting on them, as in  \eqref{eq:algnorm}). They
identify 
$\THR(A)$ with $B(A, N_e^{C_2}i_e^*A, A)$ in \cite[Theorem 2.23]{dmpr}
under a flatness assumption on $A$. 
    
The following result is an algebraic version of this result where we
use the $C_2$-action on $R \otimes R$ that exchanges the coordinates and acts
on both tensor factors.  

 \begin{thm} \label{thm:barfix}
 If $R$ is a commutative ring with involution and $2$ is invertible in
 $R$, then  there is a natural equivalence of simplicial $C_2$-Tambara functors
 \[
 \cL^{C_2}_{S^{\sigma}}(\Af) \cong 
 B (\Af, N^{C_2}_ei^{*}_e\Af, \Af) \cong
 \und{B(R, R\otimes R, R) }^\fix
 \]
 where $C_2$ acts on $R\otimes R$ as in  \eqref{eq:algnorm}.
 \end{thm}
 
So  in every simplicial degree $n$,  $\cL^{C_2}_{S^{\sigma}}(\Af)_n= \Af \Box
(N^{C_2}_ei^{*}_e\Af )^{\Box n} \Box \Af$ is the fixed point Tambara
functor of the $C_2$-ring $R \otimes (R\otimes R)^{\otimes n} \otimes
R$ with $C_2$-action given by 
\begin{align*}
\tau&  (a_0\otimes (a_1\otimes a_{2n+1})\otimes (a_2\otimes a_{2n})\otimes \cdots\otimes (a_n\otimes a_{n+2} )\otimes a_{n+1}) \\
& = \bar{a}_0\otimes (\bar{a}_{2n+1} \otimes \bar{a}_{1})\otimes (\bar{a}_{2n}\otimes \bar{a}_{2})\otimes \cdots\otimes (\bar{a}_{n+2} \otimes\bar{a}_{n} )\otimes \bar{a}_{n+1}. 
\end{align*} 

One can visualize this $C_2$-action as

\begin{center}
\setlength{\unitlength}{1cm}
\begin{picture}(9,3)
\put(0,2){$a_1$}
\put(1,2.5){$a_0$}
\put(2,2){$a_{2n+1}$}
\put(0.5,2.2){$\otimes$}
\put(2.3,1.5){$\otimes$}
\put(2.4,1){$\vdots$}
\put(0,1){$\vdots$}
\put(1.5,2.2){$\otimes$}
\put(2.3,0.6){$\otimes$}
\put(-0.1,0.6){$\otimes$}
\put(-0.1,1.5){$\otimes$}
\put(0,0.3){$a_n$}
\put(2,0.3){$a_{n+2}$}
\put(0.5,0){$\otimes$}
\put(1.9,0){$\otimes$}
\put(1,-0.3){$a_{n+1}$}
\put(3.5,1){$\mapsto$}
\put(5,2){$\bar{a}_{2n+1}$}
\put(6,2.5){$\bar{a}_0$}
\put(7,2){$\bar{a}_{1}$}
\put(5.5,2.2){$\otimes$}
\put(7.3,1.5){$\otimes$}
\put(7.4,1){$\vdots$}
\put(5,1){$\vdots$}
\put(6.5,2.2){$\otimes$}
\put(7.3,0.6){$\otimes$}
\put(4.9,0.6){$\otimes$}
\put(4.9,1.5){$\otimes$}
\put(5,0.3){$\bar{a}_{n+2}$}
\put(7,0.3){$\bar{a}_{n}$}
\put(5.5,-0.1){$\otimes$}
\put(6.9,-0.1){$\otimes$}
\put(6,-0.3){$\bar{a}_{n+1}$}
\end{picture}
\end{center}

\bigskip 

 \begin{rem}
Note that in contrast to Proposition \ref{prop:orbits} $N^{C_2}_ei^{*}_e \Af$ is
\emph{not} isomorphic to 
$\und{(R\otimes R)}^\fix$, even in very simple cases!  For example, for
 $R=\Z$ with the trivial $C_2$-action, $\und{(R\otimes R)}^\fix$ is just
 $\underline{\Z}^c$ with respect to the constant action, while
$N^{C_2}_ei^{*}_e \underline{\Z}^c$ is the $C_2$-Burnside Tambara
functor (see for instance \cite[(5.1)]{lrz-gloday}). We need an outer
copy of $\und{R}^\fix$ in Theorem \ref{thm:barfix} as a catalyst in order to
achieve the desired simplification. 
  \end{rem}

\begin{proof}
The proof follows by induction on $n$.
The base case $n=0$ is Lemma \ref{lem:fixprod} applied to $M=N=R$, and the
inductive step can be done with the help of Lemma \ref{lem:prodnorm}
for $M=  R \otimes 
(R\otimes R)^{\otimes(n-1)} \otimes R$.  Note that both lemmas proceed
by identifying all the terms to the $C_2$-coinvariant (second) part of
the box product on $C_2/C_2$, so these identifications of the term 
$\Af \Box (N^{C_2}_ei^{*}_e\Af )^{\Box n} \Box \Af$ with $\und
{(R\otimes (R\otimes R)^{\otimes n} \otimes R)} ^\fix$ behave as one
would expect for internal multiplications and insertions of units.
See also the comment 
below the proof of Lemma  \ref{lem:prodnorm}.  Therefore, these identifications are
compatible with the simplicial structure maps.  
\end{proof}

\begin{rem} \label{rem:znota}
It is important to remember that the equivariant Loday construction $\cL_{X}^{C_2}(\Af)$ is
\emph{not} the Loday construction relative to $\und{\Z}^c$, but rather
the Loday construction relative to the $C_2$-Burnside Tambara functor,
and these are different. For example, taking the relative norm-restriction term from \eqref{eq:relnorm}  $N_e^{C_2,
  \und{\Z}^c}i_e^*\und{\Z}^c $  gives $\und{\Z}^c$, whereas taking
$N_e^{C_2}i_e^*\und{\Z}^c $ gives the $C_2$-Burnside Tambara functor.

However, Theorem
\ref{thm:barfix} shows that if $2$ is invertible 
in $R$, 
 \[
 \cL^{C_2}_{S^{\sigma}}(\Af) \cong 
 B(\Af, N^{C_2}_ei^{*}_e\Af, \Af) \cong
 \und{B(R, R\otimes R, R)}^\fix
 \]
 where $C_2$ acts on $R\otimes R$ as in \eqref{eq:algnorm}.
Note that in the bar construction $B(R, R\otimes R, R)$ the ground
ring is the ring of integers whereas for the Loday construction we
work relative to the Burnside Tambara functor. Similarly, Theorem
\ref{thm:relativeisgood} implies that for $k = \Z$ we also obtain that
$\cL^{C_2, \und{\Z}^c}_{S^\sigma}(\und{R}^\fix)$ can be identified with
$\und{B^{\Z}(R, R\otimes R, R)}^\fix$ and therefore in
hindsight we obtain that in this case the Loday construction relative to the
Burnside Tambara functor agrees with the one relative to $\und{\Z}^c$. 
\end{rem}

\begin{rem}
If $R$ is a commutative ring with involution and if $M$ is an
$R$-module with involution compatible with the involution on $R$  in
the sense that $\overline{rm}=\bar{r}\bar{m}$ for all $r \in R$, $m
\in M$, then the $C_2$-Mackey functor $\und{M}^\fix$ is a symmetric
bimodule over the $C_2$-Tambara functor $\Af$.

Equivariant Loday constructions on based $G$-simplicial sets $X$ of a
$G$-Tambara
functor $\und{T}$ with coefficients in a $G$-Mackey functor $\und{N}$
which is a symmetric $\und{T}$-bimodule are defined analogously to those in the
non-equivariant case. We place the coefficients at the basepoint
in each simplicial degree. Then $\cL_X^G(\und{T}; \und{N})$ is a
simplicial $G$-Mackey functor. 
\end{rem}

\section{Relating $\cL_{S^\sigma}^{C_2}(\und{R})$  to reflexive homology} \label{sec:hr}

Let us for now consider a more general context: Let $k$ be a
commutative ring and let $A$ be an  associative $k$-algebra. We assume
that $A$ carries an anti-involution that we denote by $a\mapsto
\bar{a}$ and which we assume to be $k$-linear. Let $M$ be an
$A$-bimodule with an 
involution $m\mapsto \bar{m}$ that is compatible with the bimodule
structure over $A$ in
the sense
that $\overline{amb}=  \bar{b} \bar{m} \bar{a}$ for all $a,b\in A$,
$m\in M$. All tensor products will be over $k$ in this section, unless
otherwise indicated.

Graves \cite[Definition 1.8]{graves} defines an involution on every
level of the Hochschild complex $\CH^k_n(A; M) =M\otimes A^{\otimes n}$
by  
\begin{equation}\label{eq:rn}
r_n(m\otimes a_1\otimes a_2\otimes\cdots\otimes a_n) =\bar m \otimes \bar a_n\otimes\cdots\otimes \bar a_2\otimes \bar a_1.
\end{equation}
For the face maps of the Hochschild complex we get that $r_{n-1}\circ
d_i=d_{n-i} \circ r_n$, so these levelwise 
maps do not preserve the simplicial structure but they reverse it.
Since this relation implies that $d\circ r_n=(-1)^n r_{n-1}\circ d$,
applying $r_n$ at each level $n$ does \emph{not} induce a map on the
associated chain complexes, unless we adjust the signs.

The $C_2$-actions given by the $r_n$-maps together with the simplicial
structure maps on $\CH^k_\cdot(A;M)$ turn $\CH^k_\cdot(A;M)$ into a
functor from the crossed simplicial group $\Delta R^{op}$ in the sense
of Fiedorowicz-Loday \cite{fl} to the
category of $k$-modules. 
In \cite[Definition 1.9]{graves}, Graves defines reflexive homology as
functor homology as follows: 
\[ \hr^{+,k}_*(A;M)=\Tor_*^{\Delta R^\op}(k^*, \CH^k_\cdot(A;M)).\]
Here $k^*$ is the constant right $\Delta R^\op$-module with value $k$
at all objects.  In \cite[Definition 2.1]{graves}, he defines a
bicomplex $C_{*,*}$ which is a bi-resolution of $k^*$. With its help
he shows in \cite[Proposition 2.4]{graves} that $ \hr^{+,k}_*(A;M)$ is the
homology of the complex $\CH^k_*(A;M)/(1-r)$, where $r$ is obtained from
the maps $r_n$ of \eqref{eq:rn} by
\begin{equation}\label{eq:r}
r(m\otimes a_1\otimes\cdots\otimes a_n) 
=
(-1)^{\frac{n(n+1)}{2}} r_n(m\otimes a_1\otimes\cdots\otimes a_n) 
=  (-1)^{\frac{n(n+1)}{2}} \bar m \otimes \bar a_n\otimes\cdots\otimes  \bar a_1.
\end{equation}
With this choice of sign, the map $r$ is a chain map, so the quotient
by $1-r$ is still a chain complex. In the following we denote by
$B^k_*(A, A\otimes A^\op, M)$ the chain complex associated to the
simplicial $k$-module $B^k(A, A\otimes A^\op, M)$. 

\begin{thm}\label{thm:resolutions}
  Assume that $k$ is a commutative ring and that $A$ is an associative
  $k$-algebra with an anti-involution as above whose underlying
  $k$-module is flat. Let $M$ be an $A$-bimodule with a compatible involution
  as above, and assume that $2$ is invertible in
  $A$. Then there is  a
  $C_2$-equivariant quasi-isomorphism of chain complexes 
\begin{equation}\label{eq:sq}
B^k_*(A, A\otimes A^\op, M)\to \CH^k_*(A;M).
\end{equation}
Here the generator $\tau$ of $C_2$ acts diagonally on $B^k_*(A, A\otimes
A^\op, M)$, where the action on $A\otimes A^\op$ is given by
$\tau(a\otimes b)= \bar b\otimes \bar a$. On
the Hochschild chain complex $C_2$ acts via $r$.
\end{thm}

\begin{cor} \label{cor:hr}
Under the assumptions of Theorem \ref{thm:resolutions}, we get
homology isomorphisms  
\begin{equation}\label{eq:easy}
H_*(B^k_*(A, A\otimes A^\op, M)) \cong \HH^k_*(A;M)
\end{equation}
\begin{equation}\label{eq:hard}
H_*(B^k_*(A, A\otimes A^\op, M)^{C_2} ) \cong \hr^{+,k}_*(A;M).
\end{equation}
\end{cor}

\begin{proof}
We get the first isomorphism because the map of Theorem
\ref{thm:resolutions} is a quasi-isomorphism. It also follows from the
fact that both complexes calculate $ \Tor_*^{A\otimes A^\op }(A,M)$
because of the assumption that $A$ is flat over $k$.  Note that in the
case $M=A$ the first isomorphism also follows from the fact that the
bar construction on the left is isomorphic to the Segal-Quillen
subdivision of the Hochschild complex \cite{segal}. 

The second isomorphism follows from the fact that $2$ is invertible in
both complexes: As $2$ is invertible in $A$, the unit of $A$, $k \ra
A$ factors through $k[\frac{1}{2}]$.  We can express every level of each of
the complexes 
as the direct sum of the $+1$-eigenspace and the $-1$-eigenspace of
the action of the generator of $C_2$ on them.  Since the actions
commute with $d$, in fact each of the complexes breaks up as the
direct sum of a positive subcomplex and a negative subcomplex.

Since the quasi-isomorphism is a $C_2$-map, it preserves this
decomposition, and as it is a quasi-isomorphism,  it must be a
quasi-isomorphism on the positive and negative subcomplexes,
respectively.  That means that we get an isomorphism 
\[H_*( B^k_*(A, A\otimes A^\op, M)^{C_2}) \to H_*(\CH^k_*(A;M) ^{C_2}),\]
but since $2$ is invertible, we have a chain  isomorphism 
\[ \CH^k_*(A;M) ^{C_2} \to  \CH^k_*(A;M) _{C_2}= \CH^k_*(A;M)/(1-r),\]
and hence the claim follows with \cite[Proposition 2.4]{graves}. 
\end{proof}

\begin{proof}[Proof of Theorem \ref{thm:resolutions}]
  We consider two $A\otimes A^\op$-flat resolutions of $A$:
We use $B^k_*(A, A\otimes A^\op, A\otimes A^\op)$ with $A\otimes A^\op$
acting on the rightmost coordinate and $B^k_*(A,A,A)$ where $A^\op$ acts
on the left and $A$ on the right, as in the $\Tor$-identification of
Hochschild homology.  We let $C_2$ act on $B_*(A, A\otimes A^\op,
A\otimes A^\op)$ by acting diagonally on all the coordinates, and
denote the action of the generator on it by $\tau$.  This action is
simplicial, and therefore commutes with $d$.  We let $C_2$ act on
$B^k_*(A,A,A)$ by setting  
\[r(a_0\otimes a_1\otimes \cdots \otimes a_n\otimes a_{n+1})
  =(-1)^\frac{n(n+1)}{2} \bar  a_{n+1}\otimes \bar a_n \otimes
  \cdots\otimes \bar a_1 \otimes \bar a_0.\] 
Because of the sign adjustment, $r$ is a chain map.  We only know that
the two resolutions are flat, not that they are projective.  But any
chain map between them that covers the identity on $A$ induces an
isomorphism on $H_0$, which is the only nontrivial homology group for
both complexes, and therefore is a quasi-isomorphism. 

We define 
$f_n \colon  B^k_n(A, A\otimes A^\op, A\otimes A^\op) \to B^k_{n}(A,A,A)$ as
\begin{align*} f_n(a_0 
\otimes  &(a_1\otimes a_{2n+2})
\otimes (a_2\otimes a_{2n+1}) \otimes \cdots \otimes (a_{n+1}\otimes a_{n+2}))\\
= & a_{n+2}a_{n+3}\cdots a_{2n+1} a_{2n+2} a_0\otimes a_1\otimes a_2\otimes\cdots\otimes a_{n+1}.
\end{align*}
This is a simplicial $A\otimes A^\op$-module map, and covers the identity on
the $A$ being resolved since in level $0$ it sends $a_0\otimes
(a_1\otimes a_2)$ to $a_2 a_0 \otimes a_1$ and both of these map
down to $a_2 a_0 a_1\in A$.
This map  is not $C_2$-equivariant, but if we define $g := r\circ
f\circ\tau$, we get $g_n \colon  B^k_n(A, A\otimes A^\op, A\otimes
A^\op) \to B^k_{n}(A,A,A)$ with 
\begin{align*} g_n(a_0 
\otimes  &(a_1\otimes a_{2n+2})
\otimes (a_2\otimes a_{2n+1}) \otimes \cdots \otimes (a_{n+1}\otimes a_{n+2}))\\
= & (-1)^\frac{n(n+1)}{2}  a_{n+2}\otimes a_{n+3}\otimes \cdots \otimes a_{2n+1} \otimes a_{2n+2}\otimes  a_0  a_1  a_2 \cdots  a_{n+1}.
\end{align*}
This is not a simplicial map but it is an $A\otimes A^\op$-module map
and it is a
chain map since $r$, $f$, and $\tau$ are chain maps. Again, it covers the
identity on $A$ since on level $0$, $a_0\otimes (a_1\otimes
a_2)\mapsto a_2 \otimes a_0 a_1$ and both of these map down to $a_2
a_0 a_1\in A$.  

We now use the fact that $2$ is invertible in $A$ 
and consider the map
\[\frac {f+g}{2} \colon B^k_*(A, A\otimes A^\op, A\otimes A^\op)\to B^k_*(A,A,A),\]
which is a map of  $A\otimes A^\op$-chain complexes and covers the
identity on $A$ since  
$f$ and $g$ are such maps. This map  is also equivariant because
\[ r\circ\frac {f+g}{2}
= r\circ\frac {f+r\circ f\circ\tau}{2}
=\frac {r\circ f + f\circ\tau}{2}
=\frac {r\circ f \circ\tau + f}{2}\circ\tau
=\frac {f+g}{2}\circ\tau.
\]
So  $\frac{f+g}{2}$ is a quasi-isomorphism of flat $A\otimes A^\op$-complexes.  By Lemma \ref{lem:flatalsogood} below, if we tensor it over $A\otimes A^\op$ with the $A\otimes A^\op$-module $M$, we get a quasi-isomorphism 
\[ \frac{f+g}{2}\otimes\id_M \colon B^k_*(A, A\otimes A^\op, M)\to
  \CH^k_*(A;M). \]
This map is equivariant because it is the tensor product of two
equivariant maps. 
\end{proof}

\begin{lem}\label{lem:flatalsogood} Let $R$ be an associative ring and
  let $\phi\colon  C_*\to D_*$ be a quasi-isomorphism between two
  bounded below chain complexes of flat right $R$-modules.  Let $M$ be
  a left $R$-module.  Then 
$\phi\otimes \id_M \colon C_*\otimes_R M \to D_*\otimes_R M$ 
is a quasi-isomorphism as well.
\end{lem}
\begin{proof}
Since $\phi$ is a quasi-isomorphism,  its mapping cone, $\cone(\phi)$, 
is acyclic.  The mapping cone is also 
a bounded-below chain complex of flat right $R$-modules, so it can be
viewed as a flat resolution of the $0$-module, possibly with a
shift. We suspend it, so that $\Sigma^a \cone(\phi)$ is a non-negative
chain complex whose bottom chain group is in degree zero. 
Since flat resolutions can be used to calculate
$\Tor$,  
\[H_*(\Sigma^a\cone(\phi)\otimes_R M) \cong\Tor_{*}^R(0, M) =0\]
for all $*$.  So $\Sigma^a\cone(\phi)\otimes_R M$ and hence
$\cone(\phi)\otimes_R M =\cone(\phi\otimes_R \id_M)$ is acyclic.  But
that forces $\phi\otimes_R \id_M$ to be a 
quasi-isomorphism. 
\end{proof}

Taking our identification of $\cL_{S^\sigma}^{C_2}(\und{R})$ with
$\und{B(R, R\otimes R,R) }^\fix$ from Theorem \ref{thm:barfix}
together with Corollary \ref{cor:hr} we obtain the following
comparison result between the homology groups of the $C_2$-Loday
construction for the flip circle $S^\sigma$ and $\und{R}^\fix$ on the one
hand and the reflexive homology groups on the other hand: 

\begin{thm} \label{thm:lodayishr}
Assume that $R$ is a commutative ring with involution and that $2$ is
invertible in $R$. If the underlying abelian group of $R$ is flat over $\Z$,
then
\[ \pi_*( \cL^{C_2}_{S^\sigma}(\und{R}^\fix )(C_2/C_2)) \cong \hr^{+,\Z}_*(R,R).  \] 
\end{thm}

The relative version follows directly from Corollary \ref{cor:hr}, Theorem \ref{thm:relativeisgood}, and the identification in \eqref{eq:compare-flip}:
\begin{thm} \label{thm:lodayishrk}
Assume that $R$ is a commutative $k$-algebra with a $k$-linear
involution and that $2$ is 
invertible in $R$. If the underlying module of $R$ is flat over $k$,
then
\[ \pi_*(\cL^{C_2 ,\und{k}^c}_{S^\sigma}(\und{R}^\fix)(C_2/C_2)) \cong
  \hr^{+,k}_*(R,R).   \]  
\end{thm}

\section{Involutive Hochschild homology as a Loday construction}
\label{sec:ihh}

Involutive Hochschild cohomology was defined in \cite{braun}.
Fern\`andez-Val\`encia and Giansiracusa extended the definition to involutive
Hochschild homology. The input is an associative algebra with
anti-involution and in 
\cite{fvg} the authors work relative to a field $k$. 

A straightforward generalization of their definition \cite[Definition 3.3.1]{fvg} to arbitrary commutative ground rings is as follows:

\begin{defn}
  Let $k$ be a commutative ring, let $A$ be an associative algebra
  with anti-involution and let $M$ be an involutive $A$-bimodule. The
  involutive Hochschild homology groups of $A$ with coefficients in
  $M$ are 
  \[ \ihh_*^k(A;M) = \Tor_*^{A^{ie}}(A;M). \]
\end{defn}

Here $A^{ie}$ is the involutive enveloping algebra. As in the
classical case its role is to describe (involutive) $A$-bimodules:
There is an equivalence of categories between the category of
involutive $A$-bimodules and the category of modules over $A^{ie}$
\cite[Proposition 2.2.1]{fvg}. As a $k$-module 
\[ A^{ie} = A \otimes_k A \otimes_k k[C_2]\]
and the multiplication on $A^{ie}$ is determined by
\[ (a \otimes b \otimes \tau^i) \cdot (c \otimes d \otimes \tau^j) =
  (a \otimes b) \cdot \tau^i(c \otimes d) \otimes \tau^{i+j}.\]
Here, $\tau(c \otimes d)$ is again $\bar{d} \otimes \bar{c}$, so
\[ (a \otimes b \otimes \tau) \cdot (c \otimes d \otimes \tau^j) =
  (a\bar{d} \otimes \bar{c}b)  \otimes \tau^{1+j}.\]
Hence we can view $A^{ie}$ as a twisted group algebra $(A \otimes
A^{\op})[C_2]$. As before, every involutive algebra $A$ is an involutive
$A$-bimodule.  

Of course we know from the classical setting of Hochschild homology
that the above definition does not yield what you want if $A$ is not flat as
a $k$-module.

We obtain a comparison theorem between involutive Hochschild homology
and the homology of the $C_2$-Loday construction of the circle
$S^\sigma$ for $\und{R}^\fix$.

\begin{thm} \label{thm:ihh}
Let $R$ be a commutative ring with a $C_2$-action. Assume that $2$ is
invertible in $R$ and that the underlying abelian group of $R$ is
 flat. Then
\[ \pi_*(\cL^{C_2}_{S^{\sigma}}(\Af)(C_2/C_2)) \cong \ihh^\Z_*(R). \]
\end{thm}

And again over a general commutative $k$, there is a relative version:

\begin{thm} \label{thm:ihhk}
Let $k$ be a commutative ring and let $R$ be a commutative $k$-algebra with a $k$-linear $C_2$-action. Assume that $2$ is
invertible in $R$ and that the underlying $k$-module of $R$ is
 flat. Then
\[ \pi_*(\cL^{C_2,\und{k}^c}_{S^{\sigma}}(\Af)(C_2/C_2)) \cong \ihh^k_*(R). \]
\end{thm}

We prove Theorem \ref{thm:ihh} by comparing $\ihh_*^\Z(R;M)$ for an
involutive $R$-bimodule $M$ to the homotopy groups at the 
$C_2/C_2$-level of the simplicial Mackey functor $\und{B(R,R\otimes R,
  M)}^\fix$ where $C_2$ acts on $R\otimes R$ by $\tau(a\otimes b)
=\bar b \otimes \bar a$.  The lemmata below should be used for $k=\Z$.
The proof of Theorem \ref{thm:ihhk} is similar, just over a general
commutative ground ring $k$. 

In the following we will always assume that $R$ is a commutative $k$-algebra
with a $k$-linear $C_2$-action, that $2$ is 
invertible in $R$ and that the underlying $k$-module of $R$ is
flat over $k$.

\begin{lem} \label{lem:h0} The zeroth homotopy group of the simplicial
  $k$-module $\und{B^k(R,
    R\otimes_k R, M)}^\fix(C_2/C_2)$ is isomorphic to the zeroth
  involutive Hochschild homology group of $R$ with coefficients in $M$: 
\[ \pi_0(\und{B^k(R, R\otimes_k R, M)}^\fix(C_2/C_2)) \cong R
  \otimes_{R^{ie}} M = \ihh_0^k(R;M). \]
\end{lem}

\begin{proof}
As $2$ is invertible, taking $C_2$-fixed points is isomorphic to
taking $C_2$-coinvariants and both functors are exact. Thus we have to
identify the quotient of 
$(R \otimes M)_{C_2}$ by the bimodule action and this yields 
$(R \otimes_{R \otimes_k R} M)_{C_2}$ which is isomorphic to $(M/\{am
-ma, a \in R, m \in M\})_{C_2}$. 
By \cite[Proposition 2.4.1]{fvg}, $R \otimes_{R^{ie}} M$ is isomorphic
to the pushout of
\[ \xymatrix{M \ar[r] \ar[d] & M_{C_2} \\
   M/\{am -ma, a \in R, m \in M\}&  }\]
and this proves the claim. 
\end{proof}

\begin{lem} \label{lem:ses}
Assume that $0 \ra M_1 \ra M_2 \ra M_3 \ra 0$ is a short exact sequence of
$R^{ie}$-modules and abbreviate the simplicial $k$-module
$\und{B^k(R, R\otimes_k R,  M_i)}^\fix(C_2/C_2)$ by $BM_i$. Then we get an induced
long exact sequence on homotopy groups   
\[ \xymatrix{ \ldots \ar[r] &\pi_nBM_1 \ar[r] &
   \pi_nBM_2 \ar[r] & \pi_nBM_3 \ar`r[d]`[l]`[llld]`[dll][dll]
    \\
&     \pi_{n-1}BM_1 \ar[r] & \ldots & }
\]
\end{lem}

\begin{proof}
As we assume that $R$ is flat over $k$, tensoring with $R$
is exact, and as $2$ is invertible, taking fixed points is
exact. Therefore, in every simplicial degree $n$, the sequence
\[ 0 \ra (BM_1)_n \ra  (BM_2)_n \ra
  (BM_3)_n \ra 0\]
is short exact and hence we obtain a short exact sequence of
simplicial $k$-modules
\[ 0 \ra BM_1 \ra BM_2 \ra BM_3 \ra 0\]
which yields a long exact sequence on homotopy groups. 
\end{proof}

\begin{lem} \label{lem:proj}
Assume that $P$ is a projective $R^{ie}$-module. Then
$\pi_n\und{B^k(R, R\otimes_k R, P)}^\fix(C_2/C_2) \cong 0$ for all
positive $n$. 
\end{lem}

\begin{proof}
In the category of $R^{ie}$-modules, $R^{ie}$ is a projective
generator and every module can be written as a quotient
of a direct sum of copies of $R^{ie}$. Our construction sends
a direct sum of modules to a direct sum of simplicial objects, yielding
a direct sum of associated chain complexes. Retracts of
modules give retracts of the associated chain complexes.  It therefore
suffices to  check the claim for $P=R^{ie}$. 

If $D$ is any $k$-module with a $C_2$-action such that $2$ acts
invertibly on $D$, then there is an
isomorphism
\[ (D \otimes_k  k[C_2])^{C_2} \cong D \]
where on the left hand side we consider the diagonal $C_2$-action:
First note that $D \otimes_k  k[C_2]$ with the diagonal action is
isomorphic to $D \otimes_k  k[C_2]$ where the $C_2$-action is only on
the right-hand factor. The isomorphism $\psi \colon D \otimes_k  k[C_2]
\ra D \otimes_k  k[C_2]$ sends a generator $d \otimes \tau^i$ to
$\tau^{-i}d \otimes \tau^i$. Then, as $2$ acts invertibly, we have
\[ (D \otimes_k  k[C_2])^{C_2} \cong (D \otimes_k  k[C_2])_{C_2} = (D
  \otimes_k  k[C_2]) \otimes_{ k[C_2]}  k. \]
So in total, $(D \otimes_k  k[C_2])^{C_2} \cong D$. 

Therefore, in every simplicial degree $n$ we can identify
\[  \und{B^k_n(R, R\otimes_k R, R^{ie})}^\fix(C_2/C_2) = (R \otimes_k (R
  \otimes_k R)^{\otimes_k n} \otimes_k (R \otimes_k R \otimes_k  k[C_2]))^{C_2}\]
with $R \otimes_k (R
  \otimes_k R)^{\otimes_k n} \otimes_k (R \otimes_k R)$. But then we are left
  with the bar construction $B^k(R, R \otimes_k R, R \otimes_k R)$
  and this has trivial homotopy groups in positive degrees. 
  
\end{proof}  

\begin{prop} \label{prop:bistor}
Assume that $R$ is a commutative $k$-algebra with a $k$-linear involution such that $2$
is invertible in $R$ and assume that $M$ is an
involutive $R$-bimodule. Then
\[ \pi_*\und{B^k(R, R\otimes_k R, M)}^\fix(C_2/C_2) \cong
  \ihh^k_*(R;M). \] 
\end{prop}
\begin{proof}
  Lemmata \ref{lem:h0}, \ref{lem:ses} and \ref{lem:proj} imply that
  $\pi_*\und{B^k(R, R\otimes_k R, -)}^\fix(C_2/C_2)$ has the
  same axiomatic description as $\Tor_*^{R^{ie}}(R;-)$.
\end{proof}

\begin{proof}[Proof of Theorems \ref{thm:ihh} and \ref{thm:ihhk}]
Theorem \ref{thm:ihh} is a special case of Proposition
\ref{prop:bistor} working with $k=\Z$ (although we are working over
the $C_2$-Burnside Tambara functor, not over $\Z^c$) and with $M =
R$. Theorem \ref{thm:ihhk} is the relative version. 
\end{proof}

\begin{rem}
Graves states a
comparison result in \cite[Theorem 9.1]{graves} between reflexive
homology, $\hr^{+,k}_*(A;M)$, and involutive Hochschild homology,
$\ihh_*^k(A;M)$. The assumptions are slightly too restrictive there:
Fern\`andez-Val\`encia and Giansiracusa prove in \cite[Proposition
3.3.3]{fvg} that $\ihh^k_*(A;M) \cong \hh^k_*(A;M)_{C_2}$ if the
characteristic of the ground field is different from $2$ and Graves
shows in \cite[Proposition 2.4]{graves}, that $\hh^k_*(A;M)_{C_2} \cong
\hr^{+,k}_*(A;M)$ if $2$ is invertible in the ground ring. The assumption
on $A$ being projective as a $k[C_2]$-module comes for free if we work
over a field of characteristic different from $2$ thanks to Maschke's
theorem. For an arbitrary ring $R$, we also get that an arbitrary
$R[G]$-module $M$ is projective if $M$ is projective as an $R$-module
and if $|G|$ is invertible in $R$ \cite[Proposition 4.4]{merling}. 

\end{rem}

\begin{rem}
  If a finite group $G$ carries a homomorphism $\varepsilon \colon G \ra C_2$,
  then one can consider an associated crossed simplicial group and the
  corresponding (co)homology theory, see for instance \cite{kp,akmp}.

For an arbitrary finite group $C_2 \neq G \neq \{e\}$ without an interesting 
homomorphism to $C_2$, there is only the version of an associated
crossed simplicial group by viewing $G$ as a 
constant simplicial group because there is no meaningful way in which
$G$ can act on the simplicial category, as the automorphisms of
$\Delta$ are isomorphic to $C_2$ (see for instance \cite[Proposition
1.13]{dk}).

On the other hand, if $G$ is a group with $n$ elements $\{g_1, \ldots,
g_n\}$, we can consider the 
\emph{unreduced suspension of $G$, $SG$}. This is the graph
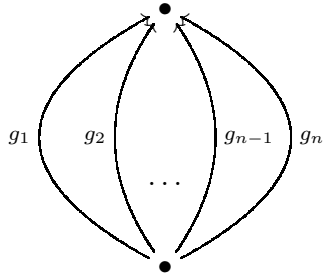
\[ \xymatrix{
    & \bullet & \\
    &  & \\
    & \ldots & \\
    & \ar@/^10ex/[uuu]^{g_1} \ar@/^4ex/[uuu]^{g_2}\ar@/_4ex/[uuu]_{g_{n-1}}
    \ar@/_10ex/[uuu]_{g_n} \bullet & }\]
and the group $G$ acts by sending an element $g \in G$ and an edge
labelled by $g_i$ to the edge $gg_i$. We can model this graph by a finite
simplicial $G$-set. 

Thus if $R$ is a commutative algebra with a $G$-action, then
\begin{equation} \label{eq:ghomology}
  \pi_*\cL^G_{SG}(\und{R}^\fix)(G/G)\end{equation}
is a perfectly fine homology theory.  
We propose \eqref{eq:ghomology} as a generalization of reflexive
homology to arbitrary finite groups, at least if $|G|$ is invertible,
and will investigate its properties in future work. 

We are grateful to the referee who pointed out that Hahn and Wilson
\cite[Question 6.3]{hw} 
suggest that $H\und{\F_p} \wedge^L_{N_e^{C_p}H\F_p} H\und{\F_p}$ might
be relevant for the Segal conjecture for the group $C_p$. Here,
$H\F_p$ is the Eilenberg-MacLane spectrum of $\F_p$ and $H\und{\F_p}$
denotes the $C_p$-equivariant Eilenberg-MacLane spectrum associated to
the constant Tambara functor on $\F_p$. The spectral analogue of our Loday
construction, $\cL^{C_p}_{SC_p}(H\und{\F_p})$, can be identified with
$H\und{\F_p} \wedge^L_{N_e^{C_p}H\F_p} H\und{\F_p} = B(H\und{\F_p}, 
{N_e^{C_p}H\F_p}, H\und{\F_p}^)$, similar to \cite[\S 7.5]{lrz-gloday}. 

\end{rem}

\section{The cases $\und{\F_2}^c$ and $\und{\Z}^c$} \label{sec:exs}

For our results we had to assume that $2$ is invertible in our
commutative ring and that the
underlying abelian group is flat. So it is a natural
question to ask what happens if we drop these assumptions. We first study the
simplest and most extreme case. 

\subsection{Comparison for $\und{\F_2}^c$}
  We consider $\F_2$ with the trivial $C_2$-action, so the fixed point
  Tambara functor is the constant Tambara functor: $\und{\F_2}^c =
  \und{\F_2}^\fix$.  Graves calculates
  reflexive homology of the ground ring in \cite[Proposition
  5.1]{graves} and in the case of $\F_2$ we obtain 
  \[ \hr^{+,\F_2}_*(\F_2) \cong H_*(BC_2, \F_2) \]
  and this is $\F_2$ in all non-negative degrees. 
  Note that here it doesn't matter whether we view $\F_2$ as a
  commutative $\F_2$-algebra or as a commutative ring (a commutative
  $\Z$-algebra). 
  Similarly, we can calculate the involutive Hochschild homology of
  $\F_2$ as an involutive $\F_2$-algebra (or as a commutative
  $\Z$-algebra) and obtain 
  \[ \ihh^{\F_2}_*(\F_2; \F_2) = \Tor_n^{\F_2[C_2]}(\F_2, \F_2) \cong
    H_*(BC_2; \F_2). \]
Hence, involutive Hochschild homology and reflexive homology agree in
this case.

\bigskip

If we compare this to the $2$-sided bar construction $B(\F_2,
\F_2\otimes \F_2, \F_2) \cong  B^{\F_2}(\F_2,
\F_2\otimes \F_2, \F_2)$, then this bar construction is isomorphic to 
the constant simplicial object with value $\F_2$ and therefore here we
obtain
\[ \pi_n\und{B(\F_2, \F_2\otimes \F_2, \F_2)}^\fix(C_2/C_2) = \pi_nB(\F_2, \F_2\otimes \F_2,
  \F_2) = \begin{cases} \F_2, & n = 0, \\ 0, &
    \text{otherwise.} \end{cases} \]
  Hence in this case $\pi_*\und{B(\F_2, \F_2\otimes \F_2,
    \F_2)}^\fix(C_2/C_2)$  agrees neither  with reflexive homology
  nor with involutive Hochschild homology.

\bigskip
What about $\pi_*B(\und{\F_2}^c, N_e^{C_2}(\F_2),
  \und{\F_2}^c)(C_2/C_2)$? 
  Note that  
  \[N_e^{C_2}(\F_2)(C_2/C_2) \cong \Z/4\Z \quad  \text{  and } 
N_e^{C_2}(\F_2)(C_2/e) \cong \F_2. \] 
In $\und{\F_2}^c \Box N_e^{C_2}(\F_2)$ we obtain:
\begin{align*}
C_2/C_2: & \quad  \left(\F_2 \otimes \Z/4\Z \oplus (\F_2 \otimes \F_2 \otimes
  \F_2)/C_2\right)/\mathrm{FR} \\
C_2/e: & \quad \F_2 \otimes \F_2 \otimes
  \F_2 \cong \F_2
\end{align*}
The $C_2$-Weyl action is trivial on $\F_2 \otimes \F_2 \otimes \F_2 \cong
\F_2$. Frobenius reciprocity yields
\[ [1 \otimes 1 \otimes 1] = [\res(1) \otimes 1 \otimes 1] \sim 1
  \otimes \trace(1 \otimes 1) = 2 \cdot 1 \otimes 1 \otimes 1 = 0\]
so at the $C_2/C_2$-level we are left with one copy of $\F_2$ and we
obtain $\und{\F_2}^c \Box N_e^{C_2}(\F_2) \cong \und{\F_2}^c$.

This identifies $B(\und{\F_2}^c, N_e^{C_2}(\F_2),
  \und{\F_2}^c)(C_2/C_2)$ with the constant simplicial object
with value $\F_2$, and therefore
\[ \pi_n\cL^{C_2}_{S^\sigma}(\und{\F_2}^c)(C_2/C_2) \cong \begin{cases} \F_2, & n = 0, \\ 0, &
    \text{otherwise.} \end{cases} \]

At the free orbit, we also get the constant simplicial object
with value $\F_2$,  and in total we get an isomorphism 
of simplicial Tambara functors between 
$\cL^{C_2}_{S^\sigma}(\und{\F_2}^c)$  and $\und{B(\F_2, \F_2\otimes
  \F_2, \F_2)}^c$. 

\subsection{Comparison for $\und{\Z}^c$}
We consider the ring of integers and this only carries a trivial
$C_2$-action. We know that norm restriction of $\und{\Z}^c$ gives the
$C_2$-Burnside Tambara functor, $N_e^{C_2}i_e^*\und{\Z}^c \cong
\und{A}$. This is the monoidal unit for the $\Box$-product.
We showed in \cite[Lemma 5.1]{lrz} that for two arbitrary commutative 
rings $A$ and $B$, $\und{A}^c \Box \und{B}^c \cong \und{(A \otimes
  B)}^c$ and hence
\[ \und{\Z}^c \Box \und{\Z}^c \cong \und{(\Z \otimes \Z)}^c \cong
  \und{\Z}^c. \]

\begin{prop}
  There is an isomorphism of simplicial $C_2$-Tambara functors
  \[ \cL_{S^\sigma}^{C_2}(\und{\Z}^c) \cong \und{\Z}^c\]
  where the right-hand side denotes the constant simplicial
  $C_2$-Tambara functor with value $\und{\Z}^c$. 
\end{prop}
\begin{proof}
  By the above arguments we get for an arbitrary simplicial degree $n$:
\begin{align*}
\cL_{S^\sigma}^{C_2}(\und{\Z}^c)_n & = \und{\Z}^c \Box
                                     (N_e^{C_2}i_e^*\und{\Z}^c)^{\Box n}
                                     \Box \und{\Z}^c \\
                                   & \cong \und{\Z}^c \Box \und{\Z}^c \\
  & \cong \und{\Z}^c.
\end{align*}
  The simplicial structure maps induce the identity maps under these
  isomorphisms. 
\end{proof}
\begin{cor}
  The homotopy groups of $\cL_{S^\sigma}^{C_2}(\und{\Z}^c)$ are
  \[ \pi_*(\cL_{S^\sigma}^{C_2}(\und{\Z}^c)) \cong \begin{cases}
      \und{\Z}^c, & *=0, \\
      0, & * > 0.\end{cases}\] 
\end{cor}

\begin{cor}
  For the $C_2$-Tambara functor $\und{\Z}^c$ the homotopy groups
  \[ \pi_*(\cL_{S^\sigma}^{C_2}(\und{\Z}^c))(C_2/C_2) \]
  are neither isomorphic to $\hr^{+,\Z}_*(\Z)$ nor to $\ihh^\Z_*(\Z)$. 
\end{cor}
\begin{proof}
We saw above that $\pi_*(\cL_{S^\sigma}^{C_2}(\und{\Z}^c))(C_2/C_2)$
is concentrated in degree $*=0$ with value $\Z$ whereas
$\hr^{+,\Z}_*(\Z)$ and $\ihh^\Z_*(\Z)$ both give $H_*(C_2; \Z)$. 
\end{proof}
\section{The case of rings and algebras with anti-involution} \label{sec:Green}
We assume now that $R$ is an associative ring with
anti-involution. In this case the fixed point Mackey functor $\und{R}^\fix$ 
is \emph{not} an associative Green functor: For $a, b \in R^{C_2}$ we get that
$\tau(ab) = \bar{b}\bar{a} = ba$ and as $a$ and $b$ do not necessarily
commute, $ab$ is not, in general, a fixed point. But $\Af$ \emph{does}
carry the structure of a discrete $E_\sigma$-ring \cite[Example
6.12]{akgh} and it is also a Hermitian Mackey functor in the
sense of \cite[Definition 1.1.1]{do}

We can still define a
replacement of the norm-restriction object that we call
$\tilde{N}_e^{C_2}i_e^*(\und{R}^\fix)$ in order to avoid confusion
with the commutative case. We claim that this can be done in the
setting of associative $C_2$-Green functors. 

\begin{defn} We define $\tilde{N}_e^{C_2}i_e^*(\und{R}^\fix)$ at the
  free level as 
  \[ \tilde{N}_e^{C_2}i_e^*(\und{R}^\fix)(C_2/e) := R \otimes R^\op \]
  and at the trivial orbit $C_2/C_2$ we define
  \[ \tilde{N}_e^{C_2}i_e^*(\und{R}^\fix)(C_2/C_2) := (\Z\{R\} \oplus (R \otimes
    R^\op)/C_2)/\mathrm{TR},  \]
  where the Tambara reciprocity relation, $\mathrm{TR}$,
  identifies $\{a+b\} \sim \{ a\} + \{b\} +[a\otimes \bar{b}]$ for all
  $a,b\in R$, just as in  
  the norm-restriction construction in the commutative case.

  The restriction map is
  \[ \res\{a\} := a \otimes \bar{a}, \qquad \res[a \otimes b] := a
    \otimes b + \bar{b} \otimes \bar{a} \]
  and the transfer sends $a \otimes b$ to
  \[ \tr(a \otimes b) := [a \otimes b]. \]
\end{defn}  
The definition above is completely analogous to the commutative case, so indeed,
this \emph{does} define a $C_2$-Mackey functor.

Note that this definition of the norm agrees with the norm defined in
\cite[Definition 3.13]{hm} as a $C_2$-Mackey functor. In \cite{hm}
Hill and Mazur 
use the notation $N^G_H$ for the norm functor for Mackey functors and
they use $\mathcal{N}_H^G$ for the norm functor for Tambara functors. 

The norm functor for Mackey functors and the restriction functor are
both strong symmetric monoidal. In our case we start with a ring
with an anti-involution and we consider the norm functor with a 
non-trivial Weyl action on the tensor factors of $R$. However, it is
straightforward to check that this additional Weyl action does not
interfere with the product:  
\begin{lem}
We can endow $\tilde{N}_e^{C_2}i_e^*(\und{R}^\fix) =
\tilde{N}_e^{C_2}R$ with the structure of an associative $C_2$-Green functor. \qed
\end{lem}

\begin{prop}
The $C_2$-Mackey functor $\und{R}^\fix$ is an  
$\tilde{N}_e^{C_2}i_e^*(\und{R}^\fix)$-bimodule. 
\end{prop}
\begin{proof}
We know that 
\[     (\tilde{N}^{C_2}_ei^{*}_e \Af) \Box \Af =
   \begin{cases}
\bigl( ( \Z\{R\} \oplus (R\otimes R^\op)/C_2 ) /\mathrm{TR}  \otimes R^{C_2} \oplus (R\otimes R^\op \otimes R)/C_2 \bigr)/\mathrm{FR}& \text{ at } C_2/C_2 \\
   R \otimes R^\op \otimes R& \text{ at } C_2/e. 
  \end{cases}
\]
We define the left $\tilde{N}^{C_2}_ei^{*}_e \Af$-module structure of $\Af$ by
\[ (a \otimes b) \otimes c \mapsto acb\]
at the free level. At the trivial level we have three types of
terms:
\begin{enumerate}
\item
  For $a \in R$ and $x \in R^{C_2}$ we send $\{a\} \otimes x$ to
  $ax\bar{a}$.
\item
Whereas for $a, b  \in R$ and $x \in R^{C_2}$ we define
  \[ [a \otimes b] \otimes x \mapsto axb +\bar{b}x\bar{a}. \]
  The resulting elements are fixed points under the anti-involution.
\item
  The $C_2$-action on $R \otimes R^\op \otimes R$ sends a generator $a
  \otimes b \otimes y$ to $\bar{b} \otimes \bar{a} \otimes
  \bar{y}$. We send a $C_2$-equivalence class $[a \otimes b \otimes
  y]$ to $ayb + \bar{b}\bar{y}\bar{a}$.
 \end{enumerate}

 We have to check that this action is well-defined and satisfies
 associativity and a unit condition.

A direct inspection shows that $[a \otimes b] \otimes x$ and
$[\bar{b} \otimes \bar{a}] \otimes x$ map to the same 
element. Similarly, the value on $[a \otimes b \otimes y]$ and
$[\bar{b} \otimes \bar{a} \otimes \bar{y}]$ agrees. It is also
straightforward to see that the module 
  structure respects Tambara reciprocity. For Frobenius reciprocity we
  have to compare three expressions:

  \begin{itemize}
\item 
$[a \otimes b] \otimes x$ is $\tr(a \otimes b) \otimes x$ and this is
identified with $[a \otimes b \otimes \res(x)] = [a \otimes b \otimes
x]$. Both terms are mapped to $axb + \bar{b}x\bar{a}$ because $x$ is a
fixed point.
\item
  A term $\{a\} \otimes \tr(y)$ is sent to $a(y + \bar{y})\bar{a}$. It
  is identified with $[\res\{a\} \otimes y] = [a \otimes \bar{a}
  \otimes y]$ and this goes to $ay\bar{a}+ a\bar{y}\bar{a}$. 
\item
  We have
  \[ [a \otimes b] \otimes (y + \bar{y}) = [a \otimes b] \otimes
    \tr(y) = \res([a \otimes b]) \otimes y. \]
  All these terms are mapped to $ayb + \bar{b}y\bar{a} + a\bar{y}b +
  \bar{b}\bar{y}\bar{a}$. 
\end{itemize}

As $\{1\}$ acts neutrally at the trivial level and as $1 \otimes 1$
acts neutrally at the free level, the unit condition is satisfied. Associativity can be proven with a very tedious calculation.

This shows that the map $\tilde{N}_e^{C_2}i_e^*\Af \Box \Af \ra \Af$ defined
above yields a left $\tilde{N}_e^{C_2}i_e^*\Af$-module structure on $\Af$. 

The following is a  sketch of the construction of the right
$\tilde{N}_e^{C_2}i_e^*\Af$-module structure on $\Af$: We have to
define $\Af \Box \tilde{N}_e^{C_2}i_e^*\Af 
\ra \Af$, that is: a map from 
\[ \Af \Box (\tilde{N}^{C_2}_ei^{*}_e \Af) =
   \begin{cases}
R^{C_2} \otimes \bigl( ( \Z\{R\} \oplus (R\otimes R^\op)/C_2 )
/\mathrm{TR} \oplus (R \otimes R\otimes R^\op)/C_2 \bigr)/\mathrm{FR}& \text{ at } C_2/C_2 \\
R \otimes R \otimes R^\op & \text{ at } C_2/e
  \end{cases}
\]
to $\Af$.
At the free level we send $a \otimes (b \otimes c)$ to $cab$ and this
propagates to the trivial level where we map $[y \otimes a \otimes b]$
to $bya + \bar{a}\bar{y}\bar{b}$ and $x \otimes [a \otimes b]$ to $bxa
+ \bar{a}x\bar{b}$. A term $x \otimes \{a\}$ goes to $\bar{a}xa$. Then
a proof dual to the above shows that this indeed gives a well-defined
right module structure and that this right-module structure is
compatible with the left-module structure so that we actually obtain a
bimodule structure. 
\end{proof}

We use this bimodule structure of
$\Af$ over $\tilde{N}_e^{C_2}i_e^*(\und{R}^\fix)$ for the definition of
$\cL_{S^\sigma}^{C_2}(\Af)$ by declaring $C_2/e \otimes \Af$ to be
$\tilde{N}_e^{C_2}i_e^*(\und{R}^\fix)$ and of course $C_2/C_2 \otimes
\Af$ is just $\Af$. As the simplices in $S^\sigma$ are lined up on two
copies of $\Delta(-,[1])$, that are just glued at the endpoints, the
associativity of $R$ suffices to obtain well-defined face and degeneracy maps
and
therefore a well-defined Loday construction $\cL_{S^\sigma}^{C_2}(\Af)$. As a
simplicial $C_2$-Mackey functor, $\cL_{S^\sigma}^{C_2}(\Af)$ is isomorphic to
$B(\Af, \tilde{N}_e^{C_2}i_e^*(\und{R}^\fix), \Af)$.

We now state and prove the analogue of Theorems \ref{thm:lodayishr}
and \ref{thm:ihh}:

\begin{thm} \label{thm:anti-invring}
Assume that $R$ is an associative ring with anti-involution and that
$2$ is invertible in $R$. If the underlying abelian group of $R$ is flat, then
\[ \ihh_*^\Z(R) \cong \pi_*( \cL^{C_2}_{S^\sigma}(\und{R}^\fix
  )(C_2/C_2)) \cong \hr^{+,\Z}_*(R,R). \]  
\end{thm}
\begin{proof}
We only point out where the differences to the proof in the
commutative case are. As in  Lemma \ref{lem:prodnorm}, we can show (by
literally using the same proof) that there is an isomorphism of
$C_2$-Mackey functors
\[ \tilde{N}_e^{C_2}i_e^* \und{R}^{\fix} \Box \und{M}^{\fix} \cong \und{(R \otimes R^{op} \otimes
  M)}^{\fix}, \]
 if $2$ is invertible in $R$ and
if $R$ is an
associative ring with anti-involution. 

The arguments in \S 5 go through with the difference that we have
to replace $R \otimes R$ by $R \otimes R^{op}$ and Theorem  \ref{thm:barfix}
gives an isomorphism  of $C_2$-Mackey functors
\[ \cL_{S^\sigma}^{C_2}(\Af) \cong B(\Af, \tilde{N}_e^{C_2}i_e^*\Af, \Af) \cong
  \und{B(R, R \otimes R^{op}, R)}^\fix. \]
Section 6 is
already formulated for associative algebras and also the
homological algebra arguments in section 7 go through but we have to replace $R
\otimes R$ by the enveloping algebra $R \otimes R^{op}$. 
\end{proof}

In the setting where we choose a commutative ring $k$ and $A$ is an
associative $k$-algebra with an anti-involution that fixes $k$, we first have to define
a relative analogue of the norm.

Note that the unit map $k \ra A$ induces a map of $C_2$-Green functors
$N_e^{C_2}i_e^*{\und{k}^c} \ra \tilde{N}_e^{C_2}i_e^*{\und{A}^\fix}$. 
\begin{defn}
  We define $\tilde{N}_e^{C_2, \und{k}^c}(\und{A}^\fix)$ as
  \[ \tilde{N}_e^{C_2, \und{k}^c}(\und{A}^\fix) :=
    \tilde{N}_e^{C_2}i_e^*{\und{A}^\fix}
    \Box_{N_e^{C_2}i_e^*{\und{k}^c}} \und{k}^c. \]
\end{defn}

With this we can define $\cL_{S^\sigma}^{C_2,\und{k}^c}(\und{A}^\fix)$ for an
associative $k$-algebra $A$ with anti-involution and obtain an isomorphism of 
simplicial $C_2$-Mackey functors
\[ \cL_{S^\sigma}^{C_2,\und{k}^c}(\und{A}^\fix) \cong B(\und{A}^\fix, \tilde{N}_e^{C_2, \und{k}^c}i_e^*(\und{A}^\fix), \und{A}^\fix).\]

We get an
analogue of Theorems  \ref{thm:lodayishrk} and \ref{thm:ihhk}. 

\begin{thm} \label{thm:anti-invalg}
Assume that $A$ is an associative $k$-algebra with a $k$-linear
anti-involution and that $2$ is 
invertible in $A$. If the underlying module of $A$ is flat over $k$,
then
\[ \ihh_*^k(A) \cong \pi_*(\cL^{C_2, \und{k}^c}_{S^\sigma}(\und{A}^\fix)(C_2/C_2))
  \cong \hr^{+,k}_*(A,A).   \] 
\end{thm}  

\begin{proof}
We have to adapt the statement and the proof of Proposition
\ref{prop:orbits} and 
claim that for an associative $k$-algebra $A$ with anti-involution we
obtain that  
\[  \tilde{N}_e^{C_2,\und{k}^c}i_e^*(\und{A}^\fix)  \cong \und{(A
    \otimes_k A^{op})}^\fix, \]  
where $C_2$ acts on $A \otimes_k A^{op}$ by $\tau(a\otimes b)=\bar b
\otimes \bar a$. 

The proof goes through, when we consider the adjunction between
the full subcategory of 
$C_2$-fixed point Mackey functors of associative rings with
anti-involution and the category of rings with anti-involution. Then
the proof of the adjunction can be copied. This 
yields an analogue of Theorem \ref{thm:relativeisgood} in the
associative setting. 
\[ \cL_{S^\sigma}^{C_2,\und{k}^c}(\und{A}^\fix) \cong
  \und{\cL_{S^\sigma}^k(A)}^\fix \] 

The other changes are similar to the absolute case of an associative
ring with anti-involution but of course now we have to replace $A \otimes_k A$ 
by the enveloping algebra $A \otimes_k A^{op}$.  

\end{proof}

\end{document}